\documentclass[11pt, reqno]{amsart}

\usepackage{lscape,amssymb, amsmath,verbatim,graphicx}
\usepackage{epsfig,subfigure,float,subfigure, color}
\usepackage{graphicx}
\usepackage{epsfig}
\newcommand{\lf}{\lfloor}
\newcommand{\rf}{\rfloor}

\allowdisplaybreaks

\newtheorem{assumption}{Assumption}[section]
\numberwithin{figure}{section}
\newtheorem{theorem}{Theorem}[section]
\newtheorem{lemma}{Lemma}[section]

\newcommand\cD{{\mathcal D}}
\newcommand\cN{{\mathcal N}}
\newcommand\fb{{\mathfrak b}}

\newcommand\eps{\epsilon}

\def\beq{\begin{equation}}
\def\eeq{\end{equation}}
\def\bals{\begin{align*}}
\def\eals{\end{align*}}

\def\bal{\begin{align}}
\def\eal{\end{align}}

\numberwithin{equation}{section}
\numberwithin{theorem}{section}
\numberwithin{table}{section}

\begin{document}

\title[$L^p$ functionals of CUSUM processes]{Limit results for $L^p$  functionals of weighted CUSUM processes}

\author {Lajos Horv\'ath}
\address{Lajos Horv\'ath, Department of Mathematics, University of Utah, Salt Lake City, UT 84112--0090 USA }

\author{Gregory Rice}
\address{Gregory Rice, Department of Statistics and Actuarial Science, University of Waterloo, Waterloo, Canada}

\thanks{  The  research of Gregory Rice was supported by the Natural Science and Engineering Research Council of Canada’s Discovery and Accelerator grants.}


\begin{abstract}
The cumulative sum (CUSUM) process is often used in change point analysis to detect changes in the mean of sequentially observed data. We provide a full description of the asymptotic distribution of  $L^p, 1\leq p <\infty$, functionals of the weighted CUSUM process for time series under general conditions.
\end{abstract}

\maketitle

\section{$L^p$ functionals of cumulative sum processes}\label{sec-main} Let $X_1, X_2, \ldots, X_N$ be a sequence of scalar observations following a simple at-most-one change point in the mean model
\begin{align*}
X_i=\left\{
\begin{array}{ll}
\mu_0+\eps_i\quad\mbox{if}\;\;\; 1\leq k \leq k^*
\vspace{.3cm}\\
\mu_A+\eps_i, \quad\mbox{if}\;\;\;k^*+1\leq k \leq N,
\end{array}
\right.
\end{align*}
where $k^*$ is the unknown change point, and $\mu_0$ and $\mu_A$ denote the means before and after the change point. To identify the mean parameters, we assume that
$$
E\eps_i=0, \;\; 1\leq i \leq N.
$$
The following developments are motivated by methods that arise in testing 
$$
H_0:\;k^*>N
$$
against the alternative of a change point in the mean
$$
H_A:\;k^*>N\;\;\;\mbox{and}\;\;\;\mu_0 \neq \mu_A.
$$

Change point detection has been an important and growing area of research in statistics and econometrics for the past several decades. For reviews we refer to Cs\"org\H{o} and Horv\'ath (1997), Aue and Horv\'ath (2013), and Horv\'ath and Rice (2014). Most statistics employed in this testing problem are based on or connected to the cumulative sum (CUSUM) process
$$
Z(x)=\sum_{i=1}^{\lf x\rf}X_i-\frac{\lf x\rf }{N}\sum_{i=1}^NX_i,\quad 0\leq x\leq N.
$$
where $\lf x\rf$ denotes the integer part of $x$.  Let
$$
Z_N(t)=N^{-1/2}Z((N+1)t/N), \quad 0\leq t\leq 1.
$$
The asymptotic properties and Gaussian approximations of $Z_N(t)$ are investigated in Cs\"org\H{o} and Horv\'ath (1997), mainly in case of independent and identically distributed $\eps_i$'s. The behaviour of $Z_N(t)$ is similar to that of empirical processes. Cs\"org\H{o} and Horv\'ath (1993) reviews results on the asymptotics of the uniform empirical and quantile processes, providing necessary and sufficient conditions for the convergence in distribution of their supremum as well as $L^p$ functionals. In change point analysis supremum functionals of $Z_N(t)$, perhaps with suitable weights applied, are often considered, since if the no change in the mean hypothesis $H_0$ is rejected, the location at which the supremum is attained can be used to estimate the time of change. However, it is well known in empirical process theory (cf.\ Shorack and Wellner (1986)) that the rate of convergence is faster for $L^p$ functionals when compared to supremum functionals.  The Cram\'er--von Mises statistic, which is the  $L^2$ functional of the standard empirical process, has received special attention in the literature. In the present note we provide limit results for $L^p$ functionals of $Z_N(t)$ under general conditions. Throughout we assume that $H_0$ holds. It follows from Chapter 3 of Billingsley (1968) that if

\beq\label{bil}
N^{-1/2}\sum_{i=1}^{\lf Nt\rf}\eps_i\;\stackrel{\cD [0,1]}{\longrightarrow}\;\sigma W(t),
\eeq
 then
\beq\label{bil2}
Z_N(t)\;\stackrel{\cD[0,1]}{\longrightarrow}\;\sigma B(t),
\eeq
where $\sigma>0$,  $\{W(t), 0\leq t \leq 1\}$ is a Wiener process and $\{B(t), 0\leq t \leq 1\}$ is a Brownian bridge. To obtain convergence of weighted functionals, we require a rate of approximation in \eqref{bil}:
\begin{assumption}\label{as-wc-2} For each $N$ there are two independent Wiener processes $\{W_{N,1}(t), 0\leq t \leq N/2\}$, $\{W_{N,2}(t), 0\leq  t \leq N/2\}$, $\sigma>0$ and $\zeta<1/2$ such that
$$
\sup_{1\leq k \leq N/2}k^{-\zeta}\left|\sum_{i=1}^{k}\eps_i-\sigma W_{N,1}(k)\right|=O_P(1)
$$
and
$$
\sup_{N/2<k< N}(N-k)^{-\zeta}\left|\sum_{i=k+1}^{N}\eps_i-\sigma W_{N,2}(N-k)\right|=O_P(1).
$$
\end{assumption}
There is a huge literature on central limit theorems and invariance principles for stationary processes; see, for example, the monographs of Ibragimov and Linnik (1971),  Bradley (2007), Dedecker et al.\ (2007) and Billingsley (1968), among others. Assumption \ref{as-wc-2} holds under a number of models allowing for dependence in the error terms, including for martingales and several versions of mixing processes. Due to their utility in applications, decomposable Bernoulli shifts are often considered. Aue et al.\ (2014) establishes
Assumption \ref{as-wc-2} for Bernoulli shifts. Optimal rates in this approximation are obtained by Berkes et al.\ (2014). \\

In order to discuss suitable weights to apply to $Z_N(t)$, we introduce the weight function $w(t)$, and consider the properties of $|Z_N(t)|^p/w(t)$. Note that given \ref{bil2}, and since $B(0)=B(1)=0$ a.s., the weight functions can be 0 only at $0$ and $1$:
\begin{assumption}\label{as-wc-1}\; $\inf_{\delta\leq t \leq 1-\delta}w(t)>0$ for all $0<\delta<1/2$. 
\end{assumption}
We aim to show that under Assumption \ref{as-wc-2} and \ref{as-wc-1} the $L^p$  functional of the weighted $Z_N(t)$ process converges in distribution if and only if
$$
P\left\{  \int_0^1\frac{|B(t)|^p}{w(t)}dt<\infty \right\}=1,
$$
which is equivalent with
\begin{assumption}\label{shao}
$$
\int_0^1\frac{(t(1-t))^{p/2}}{w(t)}dt<\infty
$$
\end{assumption}
(cf.\ Cs\"org\H{o} et al.,\ 1993).
\begin{theorem}\label{th-ad} If $p\geq 1$, $H_0$, and Assumptions  \ref{as-wc-2}--\ref{shao} are satisfied, then
\beq\label{shao-11}
\frac{1}{\sigma^p}\int_0^1\frac{|Z_N(t)|^p}{w(t)}dt\;\stackrel{\cD}{\to}\;\int_0^1\frac{|B(t)|^p}{w(t)}dt,
\eeq
where $\{B(t), 0\leq t \leq 1\}$ is a Brownian bridge.
\end{theorem}

In case of the supremum of the weighted CUSUM process, the weight function $(t(1-t))^{1/2}$ received special attention. Using $w(t)=(t(1-t))^{1/2}$ the process $|Z_N(t)|/w(t)$ is standardized to have a constant variance, and is related to the maximally selected likelihood ratio test. Since $\sup_{0<t<1}|B(t)|(t(1-t))^{-1/2}=\infty$ with probability 1, the supremum functional of the standardized CUSUM process cannot converge in distribution (cf.\ Andrews, 1993). However, a Darling--Erd\H{o}s type result can be established for the supremum functional with these weights (cf.\ Cs\"org\H{o} and Horv\'ath, 1993, 1997). In case of the $L^p$ functionals, the weight function $(t(1-t))^{1+p/2}$ plays a similar role. To state the result, we develop some additional notation. Let for $p\geq 1$,
 \begin{align*}
 a(p)=\int_{-\infty}^\infty \!\int_{-\infty}^\infty\!\int_{-\infty}^\infty &|xy|^p\Biggl\{\frac{1}{2\pi (1-\exp(-2|u|))^{1/2}}\exp\biggl(
 -\frac{1}{2(1-\exp(-2|u|))}(x^2+y^2\\
 &-2\exp(-|u|)|xy|)\biggl)-\phi(x)\phi(y)\Biggl\}dxdydu,
 \end{align*}
 and
 $$
 b(p)=\int_{-\infty}^\infty |x|^p\phi(x)dx,
 $$
 where
 $$
 \phi(x)=\frac{1}{(2\pi)^{1/2}}\exp\left(-\frac{1}{2}x^2\right)
 $$
 is the standard normal density function.
 \begin{theorem}\label{da-int} Let $p\geq 1$. If $H_0$  and  Assumption  \ref{as-wc-2}  are satisfied, then
 \begin{align}\label{int-def}
 \left(\frac{1}{4a(p)\log N}\right)^{1/2}\Biggl\{\frac{1}{\sigma^p}\int_0^1\frac{|Z_N(t)|^p}{(t(1-t))^{1+p/2}}dt-2b(p)\log N
 \Biggl\}\;\stackrel{\cD}{\to}\;\cN,
 \end{align}
 where $\cN$ denotes a standard normal random variable.
 \end{theorem}

If we use heavier weights than $(t(1-t))^{1+p/2}$ then it may be show that the integral in \eqref{int-def} is not asymptotically pivotal due to the ``heavy weights" at 0 and 1. R\'enyi (1953) suggested that in this case we should integrate over a shorter interval than $(0,1)$, and still increase the power of the test. We hence consider intervals of the form $(t_1,1-t_2)$, where $0 < t_1 < t_2 < 1$, $t_1=t_1(N)$, and $t_2=t_2(N)$.  
\begin{assumption}\label{as-re-1} (i) $\min( t_1(N), 1-t_2(N))\to 0$  and   (ii) $N\min( t_1(N), 1-t_2(N)) \to \infty$.
 \end{assumption}
 Let
 $$
 r_N=\min(t_1(N), 1-t_2(N)).
 $$
The limit is  defined in terms of the following random variables:
Let $\fb_1(p, \kappa)$ and $\fb_2(p, \kappa)$ be independent random variables,
$$
\fb_1(p, \kappa)\stackrel{\cD}{=}\fb_2(p, \kappa)\stackrel{\cD}{=}\int_1^\infty \frac{|W(t)|^p}{t^\kappa}dt.
$$
and define
$$
\fb(p, \kappa)=\gamma^{\kappa-p/2-1}_1\fb_1(p, \kappa)+\gamma^{\kappa-p/2-1}_2\fb_2(p, \kappa),
$$
where $\gamma_1$ and $\gamma_2$  are defined as
$$
\lim_{N\to\infty}\frac{r_N}{t_1(N)}=\gamma_1\quad \mbox{and}\quad \lim_{N\to\infty}\frac{r_N}{1-t_2(N)}=\gamma_2.
$$

\begin{theorem}\label{the-re-2} Let $p\geq 1$. If $H_0$, Assumptions \ref{as-wc-2} and \ref{as-re-1} are satisfied, $\kappa>p/2+1$, then
$$
r_N^{\kappa-p/2+1}\frac{1}{\sigma^p}\int_{t_1}^{t_2}\frac{|Z_N(t)|^p}{(t(1-t))^\kappa}dt\;\stackrel{\cD}{\to}\;\fb(p, \kappa).
$$
\end{theorem}

\section{Proofs}

\begin{lemma}\label{lem1} If Assumption \ref{as-wc-2} holds, we can define a sequence of Brownian bridges $\{B_N(t), 0\leq t \leq 1\}$
such that
$$
\sup_{0< t <1}\frac{1}{(t(1-t))^\zeta}|Z_N(t)-\sigma B_N(t)|=O_P(N^{-1/2+\zeta}),
$$
where $\zeta<1/2$ is defined in Assumption \ref{as-wc-2}.
\end{lemma}

\begin{proof}
We note that under the null hypothesis $Z_N(t)$ does not depend on the mean, so we need to consider the CUSUM process of the errors, the $\eps_i$'s. We write
\begin{align}\label{cu-1-5}
\sum_{i=1}^k\eps_i -\frac{k}{N}\sum_{i=1}^N\eps_i =\left\{
\begin{array}{ll}
\displaystyle \sum_{i=1}^k\eps_i-\frac{k}{N}\left(\sum_{i=1}^{\lf N/2\rf}\eps_i+\sum_{i=\lf N/2\rf+1}^N\eps_i\right),\quad\mbox{if}\;\;1\leq k \leq N/2,
\vspace{.3cm}\\
\displaystyle -\sum_{i=k+1}^N\eps_i +\frac{N-k}{N}\left(\sum_{i=1}^{\lf N/2\rf }\eps_i+\sum_{i=\lf N/2\rf+1}^N\eps_i\right),\quad\mbox{if}\;\;N/2< k < N.
\end{array}
\right.
\end{align}
 Using the Wiener processes of Assumption \ref{as-wc-2}, we define along the lines of \eqref{cu-1-5}
\begin{align}\label{cu-1-55}
\Gamma_N(x)=\left\{
\begin{array}{ll}
\displaystyle W_{N,1}(x)-\frac{x}{N}\left(W_{N,1}(N/2) + W_{N,2}(N/2) \right), \quad\mbox{if}\;\;0\leq x \leq N/2,
\vspace{.3cm}\\
\displaystyle -W_{N,2}(N-x)+\frac{N-x}{N}\left(W_{N,1}(N/2) + W_{N,2}(N/2)   \right), \quad\mbox{if}\;\;N/2\leq x \leq N
\end{array}
\right.
\end{align}
and we have
\beq\label{cu-1-6}
\max_{1\leq k \leq N}\frac{N^\zeta}{(k(N-k))^\zeta}\left|\left(\sum_{i=1}^k\eps_i -\frac{k}{N}\sum_{i=1}^N\eps_i\right)-\sigma\Gamma_N(k)  \right|=O_P(1).
\eeq

If
\beq\label{cu-1-6/7}
B_N(t)=N^{-1/2}\Gamma_N(Nt),\quad 0\leq t \leq 1,
\eeq
then for each $N$, $B_N(t)$ is a continuous Gaussian process with $EB_N(t)=0$ and $EB_N(t)B_N(s)=\min(t,s)-ts$, so it is a Brownian bridge. We note that
$$
 \sup_{0\leq t\leq 1}\sup_{|s|\leq 1/N}| B(t)-B(t+s)|=O_P((\log N)/N^{1/2})
$$
(see pg.\ 26 Cs\"org\H{o} and R\'ev\'esz, 1981). Hence
$$
\sup_{1/(N+1)\leq t \leq N/(N+1)} \frac{1}{(t(1-t))^\zeta}|Z_N(t)-\sigma B_N(t)|=O_P(N^{-1/2+\zeta}).
$$
 Using the representation $ B(t)=W(t)-tW(1)$ we get
 $$
 \sup_{0<t\leq 1/(N+1)}\frac{|B(t)|}{t^\zeta}=\sup_{0<t\leq 1/(N+1)}\frac{|W(t)|}{t^\zeta}+O_P(N^{\zeta-1})
 $$
and by the scale transformation of the Wiener process we have
$$
N^{1/2-\zeta}\sup_{0<t\leq 1/(N+1)}\frac{|W(t)|}{t^\zeta}\;\stackrel{\cD}{=}\sup_{0<t\leq N/(N+1)}\frac{|W(t)|}{t^\zeta}
$$
and therefore
$$
\sup_{0<t\leq 1/(N+1)}\frac{|B(t)|}{t^\zeta}=O_P(N^{\zeta-1/2}).
$$
By symmetry,
$$
\sup_{N/(N+1)\leq t< 1}\frac{|B(t)|}{(1-t)^\zeta}=O_P(N^{\zeta-1/2}).
$$
Since $Z_N(t)=0$ if $t\not\in [1/(N+1), N/(N+1)]$, completing the proof.
\end{proof}

\medskip
{\it Proof of Theorem  \ref{th-ad}.} We note that
\beq\label{pine}
\left||Z_N(t)|^p-|\sigma B_n(t)|^p    \right|\leq p2^p\left(|Z_N(t)-\sigma B_n(t)||\sigma B_N(t)|^{p-1} +|Z_N(t)-\sigma B_n(t)|^p      \right).
\eeq
It follows from Assumption \ref{as-wc-1} and Lemma \ref{lem1} that for any $\delta>0$,
\beq\label{no1}
\int_\delta^{1-\delta}\left|\frac{|Z_N(t)|^p-|\sigma B_N(t)|^p}{w(t)}\right|dt=o_P(1).
\eeq
Using again Lemma \ref{lem1} we get
\begin{align*}
\int_{1/(N+1)}^\delta&\frac{|Z_N(t)-\sigma B_N(t)||\sigma B_N(t)|^{p-1}}{w(t)}dt\\
&\leq
\sup_{1/(N+1)\leq t\leq N/(N+1)}\frac{|Z_N(t)-\sigma B_N(t)|}{t^\zeta}\int_{1/(N+1)}^\delta\frac{t^\zeta |\sigma B_N(t)|^{p-1}}{w(t)}dt\\
&=
O_P(N^{-1/2+\zeta})\int_{1/(N+1)}^\delta t^{\zeta-1/2}\frac{t^{1/2}|\sigma B_N(t)|^{p-1}}{w(t)}dt\\
&=
O_P(N^{-1/2+\zeta})(N+1)^{1/2-\zeta}\int_{1/(N+1)}^\delta\frac{t^{1/2}|\sigma B_N(t)|^{p-1}}{w(t)}dt.
\end{align*}
It is easy to see using Fubini's theorem that for any $q_1\geq 0$ and $q_2\geq 0$
\begin{align}\label{nomo}
E\int_{0}^\delta\frac{(t(1-t))^{{q_1}/2}| B_N(t)|^{q_2}}{w(t)}dt=E|\cN|^{q_2}\int_{1/(N+1)}^\delta\frac{(t(1-t))^{(q_1+q_2)/2}}{w(t)}dt,
\end{align}
where $\cN$ stands for a standard normal random variable. Hence for all $\varepsilon>0$
\begin{align}\label{no-2}
\lim_{\delta\to 0}\limsup_{N\to\infty}P\left\{\int_{0}^\delta \frac{|Z_N(t)-\sigma B_N(t)||\sigma B_N(t)|^{p-1}}{w(t)}dt>\varepsilon
\right\}=0
\end{align}
since $Z_N(t)=0$, if $0<t<1/(N+1)$. Similar arguments give
\begin{align}\label{no-3}
\lim_{\delta\to 0}\limsup_{N\to\infty}P\left\{\int^{1}_{1-\delta} \frac{|Z_N(t)-\sigma B_N(t)||\sigma B_N(t)|^{p-1}}{w(t)}dt>\varepsilon
\right\}=0.
\end{align}
Following our previous arguments we get
\begin{align*}
\int_{1/(N+1)}^\delta &\frac{|Z_N(t)-\sigma B_N(t)|^p}{w(t)}dt\\
&\leq \sup_{1/(N+1)<t<N/(N+1)}\left(\frac{|Z_N(t)-\sigma B_N(t)|}{t^\zeta}\right)^p
\int_{1/(N+1)}^\delta \frac{t^{p\zeta}}{w(t)}dt\\
&=O_P(N^{(-1/2+\zeta)p})(N+1)^{-p\zeta+p/2}\int_{0}^\delta \frac{t^{p/2}}{w(t)}dt.
\end{align*}
Hence by \eqref{nomo} we have for all $\epsilon>0$
\begin{align}\label{no-4}
\lim_{\delta\to 0}\limsup_{N\to\infty}P\left\{\int_{0}^\delta\frac{|Z_N(t)-\sigma B_N(t)|^p}{w(t)}dt >\varepsilon
\right\}=0.
\end{align}
By symmetry,
\begin{align}\label{no-5}
\lim_{\delta\to 0}\limsup_{N\to\infty}P\left\{\int^{1}_{1-\delta}\frac{|Z_N(t)-\sigma B_N(t)|^p}{w(t)}dt >\varepsilon
\right\}=0.
\end{align}
Since the distribution of $B_N$ does not depend on $N$, the result follows from \eqref{no1} and\eqref{no-2}--\eqref{no-5}.
\qed

The proof of Theorem \ref{da-int} is based on the following lemma:
\begin{lemma}\label{le-csh} If $p\geq 1$,  then
 \begin{align*}
 \left(\frac{1}{4a(p)\log N}\right)^{1/2}\Biggl\{\int_{1/(N+1)}^{N/(N+1)}\frac{|B(t)|^p}{(t(1-t))^{1+p/2}}dt-2b(p)\log N
 \Biggl\}\;\stackrel{\cD}{\to}\;\cN,
 \end{align*}
 where $\{B(t), 0\leq t \leq 1\}$ is a Brownian bridge and $\cN$ denotes a standard normal random variable.
\end{lemma}
\begin{proof} The proof is given in Cs\"org\H{o} and Horv\'ath (???, p.\ ???).
\end{proof}
\noindent
{Proof of Theorem \ref{da-int}.} Using again Lemma \ref{lem1} and \eqref{pine} we conclude
\begin{align*}
\int_{1/(N+1)}^{N/(N+1)}\frac{||Z_N(t)|^p-|\sigma B_N(t)|^p|}{t^{1+p/2}}dt
\leq p2^p&\int_{1/(N+1)}^{N/(N+1)}\frac{|Z_N(t)-\sigma B_N(t)||\sigma B_N(t)|^{p-1}}{t^{1+p/2}}dt\\
&+p2^p\int_{1/(N+1)}^{N/(N+1)}\frac{|Z_N(t)-\sigma B_N(t)|^p}{t^{1+p/2}}dt,
\end{align*}
\begin{align*}
\int_{1/(N+1)}^{N/(N+1)}&\frac{|Z_N(t)-\sigma B_N(t)|| B_N(t)|^{p-1}}{t^{1+p/2}}dt\\
&\leq \sup_{1/(N+1)\leq t \leq N/(N+1)}\frac{|Z_N(t)-\sigma B_N(t)|}{(t(1-t))^\zeta}\int_{1/(N+1)}^{N/(N+1)}\frac{| B_N(t)|^{p-1}}{t^{1+p/2-\zeta}}dt\\
&=O_P\left(N^{-1/2+\zeta}\right)\int_{1/(N+1)}^{N/(N+1)}\frac{(t(1-t))^{(p-1)/2}}{t^{1+p/2-\zeta}}dt\\
&=O_P(1).
\end{align*}
Similarly,
\begin{align*}
\int_{1/(N+1)}^{N/(N+1)}&\frac{|Z_N(t)-\sigma B_N(t)|^p}{t^{1+p/2}}dt\\
&\leq \sup_{1/(N+1)\leq t \leq N/(N+1)}\left(\frac{|Z_N(t)-\sigma B_N(t)|}{(t(1-t))^\zeta}\right)^p\int_{1/(N+1)}^{N/(N+1)}\frac{(t(1-t))^{(p-1)/2}}{t^{1+p/2-\zeta}}dt
\end{align*}
Since again $Z_N(t)=0$, if $t\not\in [1/(N+1), N/(N+1)]$ we get
$$
\frac{1}{\sigma^p}\int_0^1\frac{|Z_N(t)|^p}{t^{1+p/2}}dt=\int_{1/(N+1)}^{N/(N+1)}\frac{|B_N(t)|^p}{t^{1+p/2}}dt
$$
and therefore the result follows from Lemma \ref{le-csh} since the distribution of $B_N$ does not depend on $N$.
\qed

\begin{lemma}\label{leint1} If  $p\geq 1$,  Assumption  \ref{as-re-1}(i) is  satisfied and $\kappa>p/2+1$, then
$$
r_N^{\kappa-p/2+1}\int_{t_1}^{t_2}\frac{|B(t)|^p}{(t(1-t))^\kappa}dt\;\stackrel{\cD}{\to}\;\fb(p, \kappa),
$$
where $\{B(t), 0\leq t \leq 1\}$ is a Brownian bridge.
\end{lemma}
\begin{proof} We follow Horv\'ath et al.\ (2020, 2020+) where similar result is obtained for the $\sup$ norm.  \\
 We use the representation of the Brownian bridge in terms of a Wiener process $\{W(t), t\geq 1\}$,
 \begin{align*}
 B(t)=\left\{
 \begin{array}{ll}
 W(t)-tW(1),\quad\mbox{if}\;\;0\leq t \leq 1/2
 \vspace{,3cm}\\
 -(1-W(t))+(1-t)W(1),\quad\mbox{if}\;\;1/2\leq t \leq 1.
 \end{array}
 \right.
 \end{align*}
 Hence
 we get the decomposition
 $$
 \int_{t_1}^{t_2}\frac{|B(t)|^p}{(t(1-t))^\kappa}=A_1+ \cdots + A_4,
 $$
 where
 $$
 A_1=\int_{t_1}^{s_1}\frac{|B(t)|^p}{(t(1-t))^\kappa}dt,\;\;\;\;A_2=\int_{s_1}^{1/2}\frac{|B(t)|^p}{(t(1-t))^\kappa}dt,
 $$
$$
 A_3=\int_{1/2}^{s_2}\frac{|B(t)|^p}{(t(1-t))^\kappa}dt,\;\;\;\;A_4=\int_{s_2}^{t_2}\frac{|B(t)|^p}{(t(1-t))^\kappa}dt,
 $$
 with $s_1=t_1\log(1/t_1)$ and $s_2=1-(1-t_2)\log(1/(1-t_2))$.
 By the mean value theorem we have
 \begin{align*}
 \left||W(t)-tW(1)|^p-|W(t)|^p\right|&\leq p (|W(t)-tW(1)|^{p-1}+|W(t)|^{p-1})t|W(1)|\\
 &\leq p 2^p (|W(t)|^{p-1}+t^{p-1}|W(1)|^{p-1})t|W(1)|,
 \end{align*}
and therefore
 \begin{align*}
 &\left|  A_1-\int_{t_1}^{s_1}\frac{|W(t)|^p}{(t(1-t))^\kappa}dt  \right|\leq p2^p\Biggl\{|W(1)|\int_{t_1}^{s_1}\frac{t|W(t)|^{p-1}}{(t(1-t))^\kappa}dt+
 |W(1)|^{p}\int_{t_1}^{s_1}\frac{t^p}{(t(1-t))^\kappa}dt
 \Biggl\}\\
 &\hspace{1cm}=O_P\left(\max\left(s_1^{3/2+p/2-\kappa},\; t_1^{3/2+p/2-\kappa},\; t_1^{p+1-\kappa},\; s_1^{p+1-\kappa} \right) \right),
 \end{align*}
 since
 $$
 E\int_{t_1}^{s_1}\frac{t|W(t)|^{p-1}}{(t(1-t))^\kappa}dt=E|W(1)|^{p-1}\int_{t_1}^{s_1}\frac{t^{1/2+p/2}}{(t(1-t))^\kappa}dt.
 $$
 Thus we get
 \beq\label{in-de-1}
 t_1^{\kappa-p/2-1}\left|  A_1-\int_{t_1}^{s_1}\frac{|W(t)|^p}{(t(1-t))^\kappa}dt  \right|=o_P(1).
 \eeq
 Also,
 \begin{align}\label{in-de-2}
 t_1^{\kappa-p/2-1}&\left|\int_{t_1}^{s_1}\frac{|W(t)|^p}{(t(1-t))^\kappa}dt-\int_{t_1}^{s_1}\frac{|W(t)|^p}{t^\kappa}dt\right|\\
 &\leq \left|1-\frac{1}{(1-s_1)^\kappa}\right|\int_{t_1}^{s_1} t_1^{\kappa-p/2-1}\frac{|W(t)|^p}{t^\kappa}dt\notag\\
 &=o_P(1).\notag
 \end{align}
 Elementary arguments give
 $$
t_1^{\kappa-p/2-1} E\int_{s_1}^{1/2}\frac{|B(t)|^p}{(t(1-t))^\kappa}dt=E|W(1)|^pt_1^{\kappa-p/2-1}\int_{s_1}^{1/2}\frac{(t(1-t))^p}{(t(1-t))^\kappa}dt=o(1),
 $$
 so by Markov's inequality
 \beq\label{in-de-3}
 t_1^{\kappa-p/2-1}A_2=o_P(1).
 \eeq
 Putting together \eqref{in-de-1}--\eqref{in-de-3} we conclude
 \beq\label{in-de-4}
 t_1^{\kappa-p/2-1}\left|A_1+A_2-\int_{t_1}^{s_1}\frac{|W(t)|^p}{t^\kappa}dt\right|=o_P(1).
 \eeq
 One can show along the lines of the proof of \eqref{in-de-4} that
 \beq\label{in-de-5}
(1- t_2)^{\kappa-p/2-1}\left|A_3+A_4-\int_{s_2}^{t_2}\frac{|W(1)-W(t)|^p}{t^\kappa}dt\right|=o_P(1).
 \eeq
 Now the independence of $\{W(t), 0\leq t \leq 1/2\}$ and $\{W(1)-W(t), 1/2\leq t\}$ implies the independence of $\fb_1(p, \kappa)$ and $\fb_2(p, \kappa)$. By the scale transformation of the Wiener process we have
 \begin{align}\label{in-de-6}
&\left\{t_1^{\kappa-p/2-1} \int_{t_1}^{s_1}\frac{|W(t)|^p}{t^\kappa}dt,\;\;\; (1- t_2)^{\kappa-p/2-1}\int_{s_2}^{t_2}\frac{|W(1)-W(t)|^p}{t^\kappa}dt\right\}\\
&\hspace{1cm}\stackrel{\cD}{=}\;\left\{t_1^{\kappa-p/2-1} \int_{t_1}^{s_1}\frac{|W(t)|^p}{t^\kappa}dt, \;\;\;(1- t_2)^{\kappa-p/2-1}\int_{s_2}^{t_2}\frac{|W(1-t)|^p}{t^\kappa}dt\right\}\notag\\
&\hspace{1cm}\stackrel{\cD}{=}\;\left\{t_1^{\kappa-p/2-1} \int_{t_1}^{s_1}\frac{|W(t)|^p}{t^\kappa}dt, \;\;\;(1- t_2)^{\kappa-p/2-1}\int^{1-s_2}_{1-t_2}\frac{|W^*(t)|^p}{t^\kappa}dt\right\}\notag\\
&\hspace{1cm}\stackrel{\cD}{=}\left\{\int_{1}^{s_1/t_1}\frac{|W(t)|^p}{t^\kappa}dt,
\;\;\;\int_1^{(1-s_2)/(1-t_2)}\frac{|W^*(t)|^p}{t^\kappa}dt\right\} \notag\\
&\hspace{1cm}\;\to(\fb_1(p, \kappa),  \fb_2(p, \kappa))\quad \mbox{a.s.,}\notag
 \end{align}
 where $\{W^*(t), t\geq 1\}$ is a Wiener process, independent of $\{W(t), t\geq 1\}$. The  Lemma now  follows from combining \eqref{in-de-4}--\eqref{in-de-6}.
 \end{proof}

\noindent
{\it Proof of Theorem \ref{the-re-2}.} Following the proof of Theorem \ref{da-int} we get that
\begin{align*}
\int_{t_1}^{t_2}\left|\frac{\displaystyle|Z_N(t)|^p-|\sigma B_N(t)|^p}{t^{\kappa}}\right|dt
\leq p2^p&\int_{t_1}^{t_2}\frac{\displaystyle|Z_N(t)-\sigma B_N(t)||\sigma B_N(t)|^{p-1}}{t^{\kappa}}dt\\
&+p2^p\int_{t_1}^{t_2}\frac{\displaystyle|Z_N(t)-\sigma B_N(t)|^p}{t^{\kappa}}dt.
\end{align*}
Lemma \ref{lem1} yields
\begin{align*}
t_1^{\kappa-p/2+1}\int_{t_1}^{1/2}\frac{\displaystyle|Z_N(t)-\sigma B_N(t)|| B_N(t)|^{p-1}}{t^{\kappa}}dt=O_P\left( (Nt_1)^{-1/2+\zeta}  \right)=o_P(1)
\end{align*}
and
\begin{align*}
t_1^{\kappa-p/2+1}\int_{t_1}^{1/2}\frac{\displaystyle |Z_N(t)-\sigma B_N(t)|^p}{t^{\kappa}}dt=O_P\left( (Nt_1)^{(-1/2+\zeta)p}  \right)=o_P(1).
\end{align*}
Similar arguments yield
\begin{align*}
(1-t_2)^{\kappa-p/2+1}\int_{1/2}^{t_2}\left|\frac{\displaystyle|Z_N(t)|^p-|\sigma B_N(t)|^p}{(1-t)t^{\kappa}}\right|dt=o_P(1).
\end{align*}
Thus we get
$$
r_N^{\kappa-p/2+1}\int_{t_1}^{t_2}\left|\frac{\displaystyle|Z_N(t)|^p-|\sigma B_N(t)|^p}{t^{\kappa}}\right|dt=o_P(1),
$$
so the result follows from Lemma \ref{leint1}.
\qed

\end{document}